\documentclass{amsart}
\usepackage{amsmath,amssymb,graphicx,latexsym}
\usepackage{hyperref}

\newtheorem{theorem}{Theorem}
\newtheorem{lemma}[theorem]{Lemma}
\newtheorem{proposition}[theorem]{Proposition}
\newtheorem*{Hensel}{Hensel's lemma}

\theoremstyle{definition}
\newtheorem{example}[theorem]{Example}
\newtheorem{remark}[theorem]{Remark}
\newtheorem*{definition*}{Definition}
\newtheorem*{notation*}{Notation}

\newcommand{\N}{\mathbb{N}}
\newcommand{\Q}{\mathbb{Q}}
\newcommand{\Z}{\mathbb{Z}}
\newcommand{\nequiv}{\mathrel{\not\equiv}}
\newcommand{\colonequal}{\mathrel{\mathop:}=}

\DeclareMathOperator{\dens}{dens}

\newcommand{\size}[1]{\lvert #1 \rvert}
\newcommand{\seq}[1]{\cite[\href{http://oeis.org/#1}{#1}]{OEIS}}

\begin{document}

\title[Limiting density of the Fibonacci sequence]{Limiting density of the Fibonacci sequence \\ modulo powers of a prime}

\author{Nicholas Bragman}

\author{Eric Rowland}
\address{
	Department of Mathematics \\
	Hofstra University \\
	Hempstead, NY \\
	USA
}

\date{August 22, 2025}

\begin{abstract}
For a given prime $p$, we determine the limit, as $\lambda \to \infty$, of the density of residues modulo~$p^\lambda$ attained by the Fibonacci sequence.
In particular, we show that this limiting density is related to zeros in the sequence of Lucas numbers modulo $p$.
The proof uses a piecewise interpolation of the Fibonacci sequence to the $p$-adic numbers and a characterization of Wall--Sun--Sun primes $p$ in terms of the $p$-adic absolute value of a number related to the $p$-adic golden ratio.
\end{abstract}

\maketitle

\section{Introduction}\label{Introduction}

The Fibonacci sequence $F(n)_{n \geq 0}$ is defined by the initial conditions $F(0) = 0$ and $F(1) = 1$ and the recurrence
\begin{equation}\label{recurrence}
	F(n + 2) = F(n + 1) + F(n)
\end{equation}
for $n \geq 0$.
It is well known that $F(n)_{n \geq 0}$ is periodic modulo~$m$.
For example, the Fibonacci sequence modulo~$7$ is
\[
	0, 1, 1, 2, 3, 5, 1, 6, 0, 6, 6, 5, 4, 2, 6, 1, 0, 1, 1, 2, 3, 5, 1, 6, \dots
\]
with period length $16$.

Not every residue modulo~$m$ is attained by $F(n)_{n \geq 0}$.
For example, no Fibonacci number is congruent to $4$ modulo~$11$.
This suggests the following question.
What is the density
\[
	\frac{\lvert\{F(n) \bmod m : n \geq 0\}\rvert}{m}
\]
of residues that the Fibonacci sequence attains modulo~$m$?
Burr~\cite{Burr} completed the characterization of integers $m \geq 2$ for which this density is $1$ (that is, every residue is attained):
The sequence $(F(n) \bmod m)_{n \geq 0}$ contains all residues modulo~$m$ if and only if $m = 5^k m'$ for some $k \geq 0$ and some $m' \in \{2, 4, 6, 7, 14\} \cup \{3^j : j \geq 0\}$.
More recently, Dubickas and Novikas~\cite{Dubickas--Novikas} showed that for every $k \geq 1$ there exists a modulus $m$ and a sequence satisfying the Fibonacci recurrence that attains exactly $k$ residues modulo $m$.
Sanna~\cite{Sanna} studied this question for other second-order recurrences.
For a constant-recursive sequence satisfying a general second-order recurrence, Bumby~\cite{Bumby} characterized the moduli for which the residues are uniformly distributed.

In this article we study the density of residues attained by the Fibonacci sequence modulo prime powers.
Specifically, we are interested in the following.

\begin{definition*}
Let $p$ be a prime.
The \emph{limiting density} of the Fibonacci sequence modulo powers of $p$ is
\[
	\dens(p) \colonequal \lim_{\lambda \to \infty} \frac{\lvert\{F(n) \bmod p^\lambda : n \geq 0\}\rvert}{p^\lambda}.
\]
\end{definition*}

This limit exists since the density of residues attained modulo~$p^\lambda$ is bounded above by the density of residues attained modulo~$p^{\lambda - 1}$.
For example, the fact that no Fibonacci number is congruent to $4$ modulo~$11$ implies that no Fibonacci number is congruent to any of $4, 15, 26, \dots, 114$ modulo~$11^2$.
By Burr's characterization, $\dens(3) = 1$ and $\dens(5) = 1$.

Rowland and Yassawi~\cite{Rowland--Yassawi} provided a framework for determining $\dens(p)$ using a piecewise interpolation of $F(n)$ to $\Z_p$ and proved that $\dens(11) = \frac{145}{264}$.
In this article we generalize this result to give an algorithm for computing $\dens(p)$, given a prime $p$.
Our main result is Theorem~\ref{main theorem} below.
Before stating it, we need some definitions and notation.

\begin{definition*}
The \emph{period length} of the Fibonacci sequence modulo~$p$, denoted $\pi(p)$, is the smallest integer $m \geq 1$ such that $F(n + m) \equiv F(n) \mod p$ for all $n \geq 0$.
The \emph{restricted period length} of the Fibonacci sequence modulo~$p$, denoted $\alpha(p)$, is the smallest integer $m \geq 1$ such that $F(m) \equiv 0 \mod p$.
\end{definition*}

Let $L(n)_{n \geq 0}$ be the sequence of Lucas numbers, defined by $L(0) = 2$, $L(1) = 1$, and $L(n + 2) = L(n + 1) + L(n)$ for $n \geq 0$.

\begin{definition*}
Let $p$ be a prime.
We say that $i \in \{0, 1, \dots, \pi(p) - 1\}$ is a \emph{Lucas zero} (with respect to $p$) if $L(i) \equiv 0 \mod p$ and a \emph{Lucas non-zero} if $L(i) \nequiv 0 \mod p$.
\end{definition*}

For example, let $p = 7$.
We have $\pi(7) = 16$ and $\alpha(7) = 8$.
Moreover, the Lucas zeros are $i = 4$ and $i = 12$; indeed, $L(4) = 7$ and $L(12) = 322$.
In Section~\ref{section: Lucas zeros} we will explicitly identify the Lucas zeros with respect to $p$.
In particular, we will show that there are at most $2$.

Let $\nu_p(n)$ denote the $p$-adic valuation of $n$; that is, $\nu_p(n)$ is the exponent of the highest power of $p$ dividing $n$.
Let
\[
	\epsilon =
	\begin{cases}
		1	& \text{if $p \equiv 1, 4 \mod 5$} \\
		-1	& \text{if $p \equiv 2, 3 \mod 5$} \\
		0	& \text{if $p = 5$.}
	\end{cases}
\]
Theorem~\ref{main theorem} shows that $\dens(p)$ depends on the integer $e = \nu_p(F(p - \epsilon))$.
One can show that $\alpha(p)$ divides $p - \epsilon$, so $e \geq 1$.
As we discuss in Section~\ref{Wall exponent}, there are no known examples of primes for which $e \geq 2$.

\begin{theorem}\label{main theorem}
Let $p \neq 2$ be a prime, and define $e = \nu_p(F(p - \epsilon))$.
Let
\[
	N(p)
	= \left\lvert\left\{F(i) \bmod p^e : \text{$i$ is a Lucas non-zero}\right\}\right\rvert,
\]
and let $Z(p)$ be the number of Lucas zeros $i$ such that $F(i) \nequiv F(j) \mod p^e$ for all Lucas non-zeros $j$.
Then
\[
	\dens(p)
	= \frac{N(p)}{p^e}
	+ \frac{Z(p)}{2 p^{2 e - 1} (p + 1)}.
\]
\end{theorem}

In particular, $\dens(p) \in \Q$ and $\dens(p) \neq 0$.
In the case that there are no Lucas zeros with respect to $p$, we have
\[
	\dens(p) = \frac{N(p)}{p^e} = \frac{\lvert \{F(i) \bmod p^e : 0 \leq i \leq \pi(p) - 1\} \rvert}{p^e}.
\]
We refer to the expression for $\dens(p)$ in Theorem~\ref{main theorem} as an algorithm rather than a formula because we do not have a way to compute $N(p)$ and $Z(p)$ without essentially computing $F(i) \bmod p^e$ for each $i \in \{0, 1, \dots, \pi(p) - 1\}$.

Next we give several examples to show the variety of behavior that occurs.

\begin{example}
Let $p = 13$.
The period length is $\pi(13) = 28$, and the restricted period length is $\alpha(13) = 7$.
There are no Lucas zeros.
The set $\{F(0), \dots, F(27)\} \bmod 13$ is $\{0, 1, 2, 3, 5, 8, 10, 11, 12\}$.
Therefore $N(13) = 9$, and $\dens(13) = \frac{9}{13}$.
In fact, $\frac{9}{13}$ is the density of residues attained by the Fibonacci sequence modulo $13^\lambda$ for every $\lambda \geq 1$.
\end{example}

\begin{example}
Let $p = 19$, for which $\pi(19) = 18 = \alpha(19)$.
The only Lucas zero is $9$.
The set
\[
	\{F(i) \bmod 19 : \text{$0 \leq i \leq 17$ and $i \neq 9$}\} = \{0, 1, 2, 3, 5, 8, 11, 13, 16, 17, 18\}
\]
has size $N(19) = 11$ and does not contain $(F(9) \bmod 19) = 15$.
Therefore $Z(19) = 1$, and $\dens(19) = \frac{11}{19} + \frac{1}{760} = \frac{441}{760}$.
This density is the limit of the decreasing sequence $1, \frac{12}{19}, \frac{210}{361}, \frac{3981}{6859}, \frac{75621}{130321}, \dots$ of densities of residues attained modulo $19^\lambda$ for $\lambda \geq 0$.
\end{example}

\begin{example}
Let $p = 31$, for which $\pi(31) = 30 = \alpha(31)$.
The only Lucas zero is $15$, but $F(15) \equiv 21 = F(8) \mod 31$.
Therefore $Z(31) = 0$ and $\dens(31) = \frac{N(31)}{31} = \frac{19}{31}$.
\end{example}

\begin{example}
Let $p = 7$, for which $\pi(7) = 16$ and $\alpha(7) = 8$.
The Lucas zeros are $4$ and $12$.
The set
\[
	\{F(i) \bmod 7 : \text{$0 \leq i \leq 15$ and $i \neq 4, 12$}\} = \{0, 1, 2, 5, 6\}
\]
has size $N(7) = 5$ and does not contain $(F(4) \bmod 7) = 3$ or $(F(12) \bmod 7) = 4$.
Therefore $Z(7) = 2$ and $\dens(7) = \frac{5}{7} + \frac{2}{112} = \frac{41}{56}$.
Figure~\ref{p=7 tree} encodes the residues attained by the Fibonacci sequence modulo~$7^\lambda$ for $0 \leq \lambda \leq 6$.
\begin{figure}
	\includegraphics[width=\textwidth]{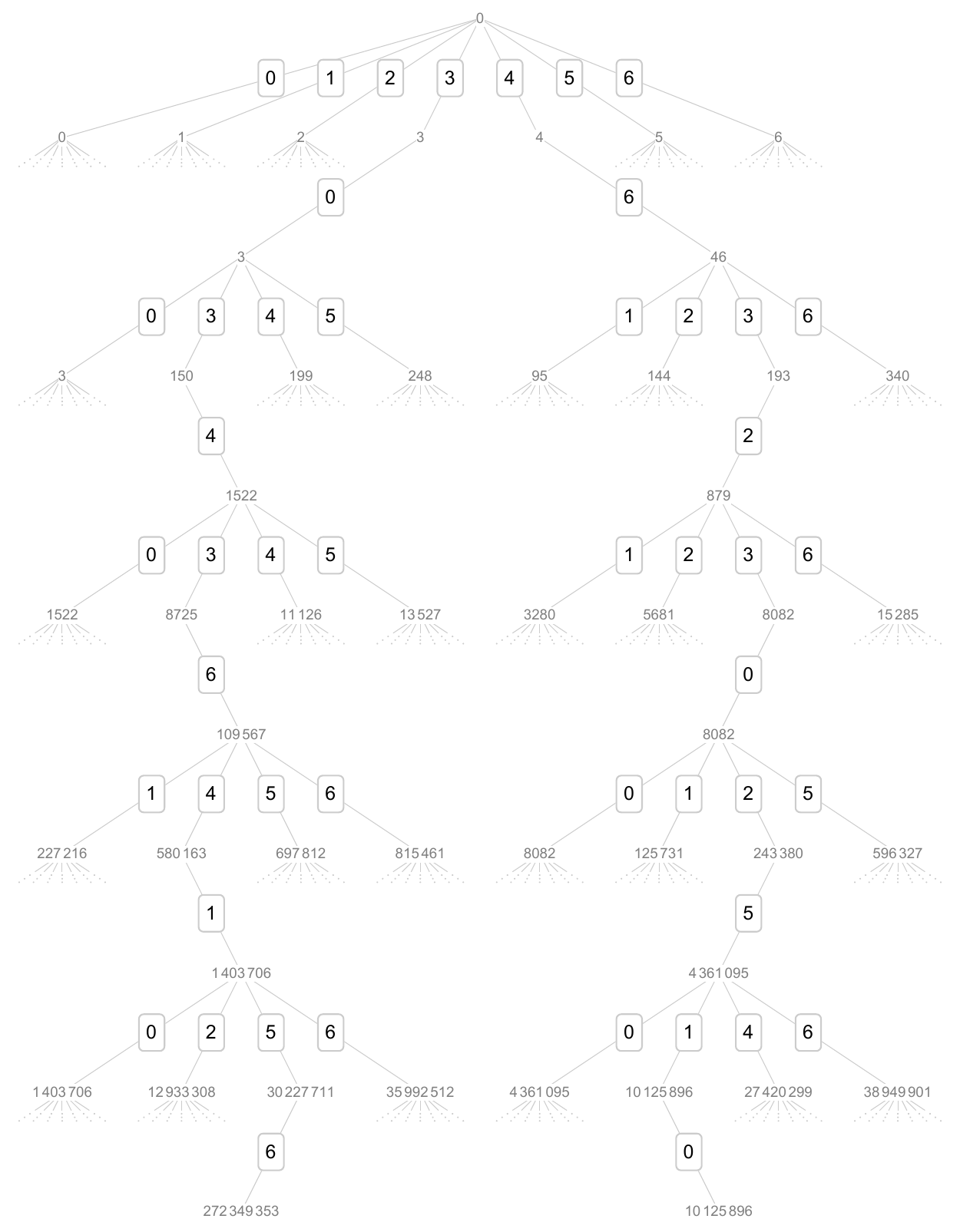}
	\caption{The tree of residues attained by the Fibonacci sequence modulo small powers of $7$. The edges are labeled with base-$7$ digits. Full $7$-ary subtrees are not shown but are indicated by dotted edges.}
	\label{p=7 tree}
\end{figure}
Level $\lambda$ in the tree contains the residues modulo~$7^\lambda$.
Dotted edges from a residue class $m$ on level $\lambda$ indicate an omitted full $7$-ary tree rooted at that vertex; that is, for every $\gamma \geq \lambda$ and for every integer $k \equiv m \mod 7^\lambda$ there exists $n \geq 0$ such that $F(n) \equiv k \mod 7^\gamma$.
On level $\lambda = 1$ the residue classes $0, 1, 2, 5, 6$ modulo~$7$ are roots of full $7$-ary trees; these trees contribute $\frac{5}{7}$ to the measure of $\overline{F(\N)}$.
We describe how to build additional levels of the tree quickly in Section~\ref{Proof of the main result}.
\end{example}

For $p = 2$, we have the following analogue of Theorem~\ref{main theorem}.

\begin{theorem}\label{theorem p=2}
The limiting density of the residues attained by the Fibonacci sequence modulo powers of $2$ is $\dens(2) = \frac{21}{32}$.
\end{theorem}

The values of $\dens(p)$ for several primes are given in the following table.
\[
	\begin{array}{r|cccccccccccccc}
		p & 2 & 3 & 5 & 7 & 11 & 13 & 17 & 19 & 23 & 29 & 31 & 37 & 41 & 43 \\
		\dens(p) & \tfrac{21}{32} & 1 & 1 & \tfrac{41}{56} & \tfrac{145}{264} & \tfrac{9}{13} & \tfrac{13}{17} & \tfrac{441}{760} & \tfrac{409}{552} & \tfrac{541}{1740} & \tfrac{19}{31} & \tfrac{29}{37} & \tfrac{715}{1722} & \tfrac{33}{43}
	\end{array}
\]
Additional values can be found in the OEIS~\cite[\href{http://oeis.org/A350999}{A350999} and \href{http://oeis.org/A351000}{A351000}]{OEIS}.
Among the first $2000$ primes, the smallest density that occurs is $\dens(9349) = \frac{504901}{174826300} \approx .002888$.
This low density can be partly explained by the fact that $F(38) \equiv 0 \mod 9349$.
An even smaller density occurs for the prime $F(29) = 514229$;
here $\dens(514229) = \frac{53}{514229} \approx .000103$.
A natural question, which we do not address here, is whether there exist primes $p$ with arbitrarily small $\dens(p)$.
More generally, what are the limit points of the set $\{\dens(p) : \text{$p$ is prime}\}$?

The proofs of Theorems~\ref{main theorem} and \ref{theorem p=2} make use of $p$-adic analytic functions; see for example Gouv\^{e}a's text~\cite{Gouvea} for the relevant background.
In Section~\ref{section: Lucas zeros} we identify the Lucas zeros with respect to $p$.
In Section~\ref{Wall exponent} we give a more convenient characterization of the exponent $e$ that appears in Theorem~\ref{main theorem}; this allows us to account for the possible existence of Wall--Sun--Sun primes.
We also give an additional characterization of $e$ as $\nu_p(L(i))$ where $i$ is a Lucas zero.
In Section~\ref{p-adic} we interpolate the Fibonacci sequence to the field of $p$-adic numbers for $p \neq 2$.
In Section~\ref{Proof of the main result} we prove Theorem~\ref{main theorem}, and in Section~\ref{p = 2} we treat the case $p = 2$.

\section{Lucas zeros}\label{section: Lucas zeros}

In this section we identify the Lucas zeros for the Fibonacci sequence with respect to $p$, in terms of $\alpha(p)$.
First we recall the following special case of a theorem of Vinson~\cite[Theorem~2]{Vinson}, which relates the period length to the restricted period length.

\begin{theorem}\label{ratio}
Let $p \neq 2$ be a prime.
Then
\[
	\frac{\pi(p)}{\alpha(p)} =
	\begin{cases}
		4	& \text{if $\alpha(p)$ is odd} \\
		1	& \text{if $\alpha(p)$ is even but not divisible by $4$} \\
		2	& \text{if $\alpha(p)$ is divisible by $4$}.
	\end{cases}
\]
\end{theorem}

(For $p = 2$ we have $\pi(2) = 3 = \alpha(2)$.)

\begin{remark}
The primes for which $\frac{\pi(p)}{\alpha(p)} = 4$ are
\begin{align*}
	& 5, 13, 17, 37, 53, 61, 73, 89, 97, 109, 113, 137, 149, \dots \qquad \text{\seq{A053028}}.
\intertext{Excluding $2$, the primes for which $\frac{\pi(p)}{\alpha(p)} = 1$ are}
	& 11, 19, 29, 31, 59, 71, 79, 101, 131, 139, \dots \qquad \text{\seq{A053032}}.
\intertext{The primes for which $\frac{\pi(p)}{\alpha(p)} = 2$ are}
	& 3, 7, 23, 41, 43, 47, 67, 83, 103, 107, 127, \dots \qquad \text{\seq{A053027}}.
\end{align*}
This ``trichotomy'' was described by Ballot and Elia~\cite{Ballot--Elia}.
\end{remark}

The number of Lucas zeros depends on the $2$-adic valuation of $\alpha(p)$ and corresponds to the three cases of Theorem~\ref{ratio}.

\begin{proposition}\label{Lucas zeros}
Let $p \neq 2$ be a prime.
The set of Lucas zeros with respect to $p$ is
\[
	\begin{cases}
		\{\}								& \text{if $\alpha(p)$ is odd} \\
		\{\frac{\alpha(p)}{2}\}					& \text{if $\alpha(p)$ is even but not divisible by $4$} \\
		\{\frac{\alpha(p)}{2}, \frac{3 \alpha(p)}{2}\}	& \text{if $\alpha(p)$ is divisible by $4$}.
	\end{cases}
\]
Moreover, if $\alpha(p)$ is divisible by $4$ then $F(\frac{\alpha(p)}{2}) \nequiv F(\frac{3 \alpha(p)}{2}) \mod p$.
\end{proposition}

Throughout the article we denote $\phi = \frac{1 + \sqrt 5}{2}$ and $\bar\phi = \frac{1 - \sqrt 5}{2} = -\frac{1}{\phi}$.
Recall Binet's formula
\begin{equation}\label{Binet}
	F(n) = \frac{\phi^n - \bar\phi^n}{\sqrt{5}}
\end{equation}
and the analogous formula $L(n) = \phi^n + \bar\phi^n$.
Together these imply $2 \phi^n = L(n) + \sqrt{5} F(n)$ and $2 \bar\phi^n = L(n) - \sqrt{5} F(n)$ as well as the fundamental identity $L(n)^2 - 5 F(n)^2 = (-1)^n 4$.
In Section~\ref{Wall exponent} we will assume $p \neq 5$ and interpret $\sqrt{5}$ as a $p$-adic number or $p$-adic algebraic number, depending whether $\Z_p$ contains a square root of $5$.
For the proof of Proposition~\ref{Lucas zeros}, it suffices to work modulo~$p$, so we work in the residue field
\[
	\begin{cases}
		\Z / (p \Z)					& \text{if $x^2 \equiv 5 \mod p$ has an integer solution} \\
		\Z[\sqrt{5}] / (p \Z[\sqrt{5}])		& \text{otherwise}.
	\end{cases}
\]

\begin{proof}[Proof of Proposition~\ref{Lucas zeros}]
The statement holds for $p = 5$, since $\alpha(5) = 5$ and no Lucas number is divisible by $5$.
Assume $p \neq 5$.

By Binet's formula, $F(n) \equiv 0 \mod p$ if and only if $\phi^n - \bar\phi^n \equiv 0 \mod p$, which can be rewritten $(-\phi^2)^n \equiv 1 \mod p$ since $\phi \bar\phi = -1$.
Since $F(\alpha(p)) \equiv 0 \mod p$ by definition, we have $(-\phi^2)^{\alpha(p)} \equiv 1 \mod p$.
Moreover, $\alpha(p)$ is the smallest positive integer with this property.
In other words, $\alpha(p)$ is the order of $-\phi^2$ modulo~$p$.

Similarly, $L(i) = \phi^i + \bar\phi^i = \frac{\phi^{2 i} + (-1)^i}{\phi^i}$, so $i$ is a Lucas zero if and only if $(-\phi^2)^i \equiv -1 \mod p$.
Therefore the Lucas zeros are precisely the values of $i \in \{0, \dots, \pi(p) - 1\}$ for which the order of $(-\phi^2)^i$ is $2$.
If $\alpha(p)$ is odd, then there are none.
If $\alpha(p)$ is even, then the Lucas zeros are the solutions to $i \equiv \frac{\alpha(p)}{2} \mod \alpha(p)$.
If $\alpha(p)$ is even but not divisible by $4$, then $\pi(p) = \alpha(p)$ by Theorem~\ref{ratio}, so the only solution is $i = \frac{\alpha(p)}{2}$.
If $\alpha(p)$ is divisible by $4$, then $\pi(p) = 2 \alpha(p)$, so the solutions are $i = \frac{\alpha(p)}{2}$ and $i = \frac{\alpha(p)}{2} + \alpha(p)$.

It remains to show that $F(\frac{\alpha(p)}{2}) \nequiv F(\frac{3 \alpha(p)}{2}) \mod p$ when $\alpha(p)$ is divisible by~$4$.
Vinson~\cite[Lemma~1]{Vinson} established that the order of $F(\alpha(p) - 1)$ modulo~$p$ is $\frac{\pi(p)}{\alpha(p)}$.
Since $\frac{\pi(p)}{\alpha(p)} = 2$, this implies that $F(\alpha(p) - 1) \equiv -1 \mod p$.
We have $F(\alpha(p)) \equiv 0 = -F(0) \mod p$, so it follows from $F(\alpha(p) + 1) = F(\alpha(p)) + F(\alpha(p) - 1)$ that
\[
	F(\alpha(p) + 1) \equiv F(\alpha(p) - 1) \equiv -1 = -F(1) \mod p.
\]
Now the Fibonacci recurrence~\eqref{recurrence} inductively implies $F(\alpha(p) + n) \equiv -F(n) \mod p$ for all $n \geq 0$.
In particular, $F(\frac{3 \alpha(p)}{2}) \equiv -F(\frac{\alpha(p)}{2}) \mod p$.
By the minimality of $\alpha(p)$, we have $F(\frac{\alpha(p)}{2}) \nequiv 0 \mod p$, so $F(\frac{3 \alpha(p)}{2}) \nequiv F(\frac{\alpha(p)}{2}) \mod p$.
\end{proof}

\section{The Wall exponent}\label{Wall exponent}

In this section we give two additional characterizations of the exponent $e = \nu_p(F(p - \epsilon))$ that appears in Theorem~\ref{main theorem}.
We will need these characterizations for the proof of Theorem~\ref{main theorem} in Section~\ref{Proof of the main result}, since we will want to express $e$ in terms of the $p$-adic absolute value of a certain $p$-adic number involving $\phi$.
First we introduce some notation.
As in Section~\ref{Introduction}, for $p \neq 5$, let
\[
	\epsilon =
	\begin{cases}
		1	& \text{if $p \equiv 1, 4 \mod 5$} \\
		-1	& \text{if $p \equiv 2, 3 \mod 5$.}
	\end{cases}
\]
Equivalently, we can write $\epsilon = (\frac{p}{5})$ using the Legendre symbol.

For primes $p \equiv 1, 4 \mod 5$, the set of $p$-adic integers $\Z_p$ contains a square root of $5$; this can be shown using quadratic reciprocity and Hensel's lemma.
For other primes, $\Z_p$ does not contain a square root of $5$, so we consider the extension $\Q_p(\sqrt{5})$.
Throughout the article, we denote
\[
	K
	= \Q_p(\sqrt{5})
	= \begin{cases}
		\Q_p			& \text{if $p \equiv 1, 4 \mod 5$} \\
		\Q_p(\sqrt{5})	& \text{if $p \equiv 2, 3 \mod 5$} \\
		\Q_5(\sqrt{5})	& \text{if $p = 5$}.
	\end{cases}
\]
The $p$-adic absolute value of $x \in K$ is denoted $\lvert x \rvert_p$; it extends the $p$-adic absolute value on $\Q$~\cite[Definition~2.1.4 and Theorem~5.3.5]{Gouvea}.
The $p$-adic valuation $\nu_p(x)$ of $x \in \Q_p$ is related to its $p$-adic absolute value by $\size{x}_p = p^{-\nu_p(x)}$.
The ring of integers of $K$ is
\[
	\mathcal O_K
	= \{x \in K : |x|_p \leq 1\}
	= \begin{cases}
		\Z_p			& \text{if $p \equiv 1, 4 \mod 5$} \\
		\Z_p[\sqrt{5}]	& \text{if $p \equiv 2, 3 \mod 5$ and $p \neq 2$} \\
		\Z_2[\phi]		& \text{if $p = 2$} \\
		\Z_5[\sqrt{5}]	& \text{if $p = 5$}.
	\end{cases}
\]
We now consider $\phi$ and $\bar\phi$ as elements of $\mathcal O_K$.

Let $d$ be the degree of the extension $K$ of $\Q_p$, and let $p^f$ be the size of the residue field $\mathcal O_K/\{x \in K : |x|_p < 1\}$.
Namely,
\[
	\text{
		$d
		= \begin{cases}
			1	& \text{if $p \equiv 1, 4 \mod 5$} \\
			2	& \text{if $p \equiv 2, 3 \mod 5$} \\
			2	& \text{if $p = 5$}
		\end{cases}$
		\qquad
		and
		\qquad
		$f
		= \begin{cases}
			1	& \text{if $p \equiv 1, 4 \mod 5$} \\
			2	& \text{if $p \equiv 2, 3 \mod 5$} \\
			1	& \text{if $p = 5$}.
		\end{cases}$
	}
\]
The quotient $d/f$ is the \emph{ramification index} of the extension.

To obtain another description of $e = \nu_p(F(p - \epsilon))$, we will use the fact that the ring of integers $\mathcal O_K$ contains the cyclic group of $(p^f - 1)$th roots of unity~\cite[Corollary 5.4.9]{Gouvea}.
Moreover, for each $x \in K$ such that $\lvert x \rvert_p = 1$, there is a unique $(p^f - 1)$th root of unity congruent to $x$ modulo~$p$ (or, in the case $p = 5$, modulo a uniformizer).
This is because there are precisely $p^f$ residue classes modulo~$p$ (or modulo the uniformizer) and each $(p^f - 1)$th root of unity belongs to a distinct residue class.

\begin{notation*}
If $\lvert x \rvert_p = 1$ and $p \neq 5$, define $\omega(x)$ to be the $(p^f - 1)$th root of unity satisfying $\omega(x) \equiv x \mod p$.
In particular, $\frac{x}{\omega(x)} \equiv 1 \mod p$.
For $p = 5$, we use the uniformizer $\sqrt{5}$ and define $\omega(x)$ to be the $4$th root of unity satisfying $\omega(x) \equiv x \mod \sqrt{5}$.
\end{notation*}

The first main result of this section is that the Wall exponent $e = \nu_p(F(p - \epsilon))$ can also be defined by $\big\lvert\frac{\phi}{\omega(\phi)} - 1\big\rvert_p = \frac{1}{p^e}$.

\begin{theorem}\label{Wall exponent characterization}
Let $p$ be a prime such that $p \neq 2$ and $p \neq 5$.
Then $\lvert F(p - \epsilon) \rvert_p = \big\lvert\frac{\phi}{\omega(\phi)} - 1\big\rvert_p$.
\end{theorem}

The prime $p$ is a \emph{Wall--Sun--Sun prime} if $\nu_p(F(p - \epsilon)) \geq 2$.
Wall--Sun--Sun primes are also known as \emph{Fibonacci--Wieferich primes}.
Wall~\cite{Wall} verified that there are none $< 10^4$.
No Wall--Sun--Sun primes are known, although a heuristic argument suggests there are infinitely many~\cite[Section~4]{McIntosh--Roettger}.
For $p \notin \{2, 5\}$, Theorem~\ref{Wall exponent characterization} implies that $p$ is a Wall--Sun--Sun prime if and only if $\big\lvert\frac{\phi}{\omega(\phi)} - 1\big\rvert_p \leq \frac{1}{p^2}$.

We split the proof of Theorem~\ref{Wall exponent characterization} into the following two propositions.
The first is a straightforward generalization of the statement that $p$ is a Wall--Sun--Sun prime if and only if $L(p - \epsilon) \equiv 2 \epsilon \mod p^4$, which is one of several characterizations of Wall--Sun--Sun primes given by McIntosh and Roettger~\cite[Theorem~1]{McIntosh--Roettger} (in a more general form).

\begin{proposition}
Let $p$ be a prime such that $p \neq 2$ and $p \neq 5$.
Then $\lvert F(p - \epsilon) \rvert_p = \sqrt{\lvert L(p - \epsilon) - 2 \epsilon \rvert_p}$.
\end{proposition}

\begin{proof}
We use the identity $5 F(n)^2 = L(n)^2 - (-1)^n 4$.
Note that $p - \epsilon$ is even and $(-1)^{(p - \epsilon)/2} = \epsilon$.
Therefore
\[
	5 F(p - \epsilon)^2 = \left(L(p - \epsilon) + 2 \epsilon\right) \left(L(p - \epsilon) - 2 \epsilon\right).
\]
McIntosh and Roettger~\cite[Equation~(5)]{McIntosh--Roettger} show that $L(p - \epsilon) - 2 \epsilon \equiv 0 \mod p^2$.
Since $4 \epsilon \nequiv 0 \mod p$, this implies $L(p - \epsilon) + 2 \epsilon \nequiv 0 \mod p$.
It follows that $\size{F(p - \epsilon)}_p^2 = \size{L(p - \epsilon) - 2 \epsilon}_p$ as desired.
\end{proof}

To prove the second proposition, we use the following lemma.

\begin{lemma}\label{Fermat quotient}
Let $p$ be a prime, and let $x \in K$ such that $\size{x}_p = 1$.
Then $\lvert x^{p^f} - x \rvert_p = \lvert x - \omega(x) \rvert_p$.
\end{lemma}

\begin{proof}
We use the factorization $x^{p^f} - x = x \prod_r (x - r)$ in $\mathcal O_K[x]$, where $r$ ranges over the $(p^f - 1)$th roots of unity.
Since $\size{x}_p = 1$, there is precisely one root of unity $r$ such that $\size{x - r}_p \neq 1$, namely $r = \omega(x)$.
Therefore $\lvert x^{p^f} - x \rvert_p = \lvert x - \omega(x) \rvert_p$.
\end{proof}

\begin{proposition}
Let $p$ be a prime such that $p \neq 2$ and $p \neq 5$.
Then $\sqrt{\lvert L(p - \epsilon) - 2 \epsilon \rvert_p} = \big\lvert\frac{\phi}{\omega(\phi)} - 1\big\rvert_p$.
\end{proposition}

\begin{proof}
First we note that $\left\lvert\bar\phi^{p - \epsilon} - \epsilon\right\rvert_p = \left\lvert\phi^{p - \epsilon} - \epsilon\right\rvert_p$, since
\[
	\left\lvert\bar\phi^{p - \epsilon} - \epsilon\right\rvert_p
	= \left\lvert\left(-\tfrac{1}{\phi}\right)^{p - \epsilon} - \epsilon\right\rvert_p
	= \left\lvert (-1)^{p - \epsilon} - \epsilon \phi^{p - \epsilon}\right\rvert_p
	= \left\lvert 1 - \epsilon \phi^{p - \epsilon}\right\rvert_p
	= \left\lvert\phi^{p - \epsilon} - \epsilon\right\rvert_p
\]
(using the fact that $p - \epsilon$ is even).
Since $L(n) = \phi^n + \bar\phi^n$, we have
\begin{align*}
	-\epsilon (\phi^{p - \epsilon} - \epsilon) (\bar\phi^{p - \epsilon} - \epsilon)
	&= -\epsilon (-1)^{p - \epsilon} + \phi^{p - \epsilon} + \bar\phi^{p - \epsilon} - \epsilon \\
	&= L(p - \epsilon) - 2 \epsilon.
\end{align*}
Taking the $p$-adic absolute value of both sides gives
$\left\lvert\phi^{p - \epsilon} - \epsilon\right\rvert_p^2 = \left\lvert L(p - \epsilon) - 2 \epsilon \right\rvert_p$.
Therefore it suffices to show that $\left\lvert\phi^{p - \epsilon} - \epsilon\right\rvert_p = \big\lvert\frac{\phi}{\omega(\phi)} - 1\big\rvert_p$.

If $p \equiv 1, 4 \mod 5$, then $\epsilon = 1$ and $f = 1$.
It follows from Lemma~\ref{Fermat quotient} that
$
	\lvert \phi^{p - 1} - 1 \rvert_p
	= \lvert \phi^p - \phi \rvert_p
	= \lvert \phi - \omega(\phi) \rvert_p
	= \big\lvert\frac{\phi}{\omega(\phi)} - 1\big\rvert_p
$.

If $p \equiv 2, 3 \mod 5$, then $\epsilon = -1$ and $f = 2$.
The geometric series formula gives
\[
	\frac{\phi^{p^2 - 1} - 1}{\phi^{p + 1} + 1}
	= \sum_{j = 0}^{p - 2} (-1)^{j + 1} \phi^{(p + 1) j}
	\equiv \sum_{j = 0}^{p - 2} (-1)^{j + 1} (-1)^j
	\equiv 1 \mod p.
\]
It follows that $\lvert \phi^{p + 1} + 1 \rvert_p = \lvert \phi^{p^2 - 1} - 1 \rvert_p = \lvert \phi^{p^2} - \phi \rvert_p$.
By Lemma~\ref{Fermat quotient},
$
	\lvert \phi^{p + 1} + 1 \rvert_p
	= \lvert \phi - \omega(\phi) \rvert_p
	= \big\lvert\frac{\phi}{\omega(\phi)} - 1\big\rvert_p
$.
\end{proof}

This concludes the proof of Theorem~\ref{Wall exponent characterization}.

The second main result of this section is that the Wall exponent is also the $p$-adic valuation of certain Lucas numbers.

\begin{theorem}\label{Wall exponent Lucas characterization}
Let $p$ be a prime such that $p \neq 2$, $p \neq 3$, and $p \neq 5$.
If $i$ is a Lucas zero with respect to $p$, then $\size{L(i)}_p = \size{F(p - \epsilon)}_p$.
\end{theorem}

For $p = 3$, the Lucas zeros are $2$ and $6$.
The Lucas zero $2$ does satisfy the conclusion, since $L(2) = 3 = F(4)$.
However, $6$ does not, since $\size{L(6)}_3 = \size{18}_3 = \frac{1}{9} < \frac{1}{3} = \size{F(4)}_3$.

\begin{proof}[Proof of Theorem~\ref{Wall exponent Lucas characterization}]
We show that $\size{L(i)}_p = \size{F(\alpha(p))}_p$.
The result will then follow from $\size{F(\alpha(p))}_p = \size{F(p - \epsilon)}_p$.

Since $i$ is a Lucas zero, Proposition~\ref{Lucas zeros} implies that $\alpha(p)$ is even.
Therefore
\[
	\size{F(\alpha(p))}_p
	= \left\lvert\frac{\phi^{\alpha(p)} - (-\phi^{-1})^{\alpha(p)}}{\sqrt{5}}\right\rvert_p
	= \left\lvert\frac{\phi^{\alpha(p)} - \phi^{-\alpha(p)}}{\sqrt{5}}\right\rvert_p
	= \big\lvert \phi^{2 \alpha(p)} - 1 \big\rvert_p.
\]
We consider two cases.

If $\alpha(p)$ is not divisible by $4$, then $i = \frac{\alpha(p)}{2}$ is odd.
Therefore
\[
	\size{L(i)}_p
	= \left\lvert\phi^i + (-\phi^{-1})^i\right\rvert_p
	= \left\lvert\phi^i - \phi^{-i}\right\rvert_p
	= \big\lvert \phi^{\alpha(p)} - 1 \big\rvert_p
	= \frac{
		\left\lvert \phi^{2 \alpha(p)} - 1 \right\rvert_p
	}{
		\left\lvert \phi^{\alpha(p)} + 1 \right\rvert_p
	}.
\]
Since $i$ is a Lucas zero, we have $\phi^{\alpha(p)} \equiv 1 \nequiv -1 \mod p$, which implies $\left\lvert \phi^{\alpha(p)} + 1 \right\rvert_p = 1$.
It follows that $\size{L(i)}_p = \size{F(\alpha(p))}_p$.

If $\alpha(p)$ is divisible by $4$, then $i \in \{\frac{\alpha(p)}{2}, \frac{3 \alpha(p)}{2}\}$, so $i$ is even.
Therefore
\[
	\size{L(i)}_p
	= \left\lvert\phi^i + (-\phi^{-1})^i\right\rvert_p
	= \left\lvert\phi^i + \phi^{-i}\right\rvert_p
	= \left\lvert \phi^{2 i} + 1 \right\rvert_p.
\]
If $i = \frac{\alpha(p)}{2}$, then
\[
	\size{L(i)}_p
	= \big\lvert \phi^{\alpha(p)} + 1 \big\rvert_p
	= \frac{
		\left\lvert \phi^{2 \alpha(p)} - 1 \right\rvert_p
	}{
		\left\lvert \phi^{\alpha(p)} - 1 \right\rvert_p
	}
	= \size{F(\alpha(p))}_p
\]
since $\size{L(i)}_p < 1$ implies $\left\lvert \phi^{\alpha(p)} - 1 \right\rvert_p = 1$.
On the other hand, if $i = \frac{3 \alpha(p)}{2}$, then
\[
	\size{L(i)}_p
	= \big\lvert \phi^{3 \alpha(p)} + 1 \big\rvert_p
	= \big\lvert \phi^{\alpha(p)} + 1 \big\rvert_p \big\lvert \phi^{2 \alpha(p)} - \phi^{\alpha(p)} + 1 \big\rvert_p.
\]
Since $\phi^{\alpha(p)} \equiv -1 \mod p$, we have $\phi^{2 \alpha(p)} - \phi^{\alpha(p)} + 1 \equiv 3 \nequiv 0 \mod p$ since $p \neq 3$.
It follows that $\size{L(i)}_p = \size{F(\alpha(p))}_p$.
\end{proof}

\section{Piecewise interpolation to the $p$-adic numbers}\label{p-adic}

In this section we interpolate $F(n)$ to $\Z_p$ so that we can work with continuous functions in Section~\ref{Proof of the main result}.
This same approach was used to prove the celebrated Skolem--Mahler--Lech theorem concerning the set of zeros of a constant-recursive sequence~\cite{Skolem, Mahler, Lech}.

The Fibonacci sequence cannot be interpolated to $\Z_p$ by a single continuous function, but there exists a finite set of continuous functions that comprise a piecewise interpolation to $\Z_p$.
These functions come from Binet's formula~\eqref{Binet} interpreted in $\Q_p(\sqrt 5)$.
For this, we need the $p$-adic logarithm and exponential functions, which are defined by their usual power series
\[
	\log_p(1 + x) \colonequal \sum_{m \geq 1} (-1)^{m + 1} \frac{x^m}{m}
	\qquad
	\text{and}
	\qquad
	\exp_p x \colonequal \sum_{m \geq 0} \frac{x^m}{m!}.
\]
The series $\log_p(1 + x)$ converges if $\lvert x \rvert_p < 1$, and $\exp_p x$ converges if $\lvert x \rvert_p < p^{-1/(p - 1)}$.
Moreover, $\log_p$ is an isomorphism from the multiplicative group $\{x : \lvert x - 1 \rvert_p < p^{-1/(p - 1)}\}$ to the additive group $\{x : \lvert x \rvert_p < p^{-1/(p - 1)}\}$, and its inverse map is $\exp_p$~\cite[Proposition~4.5.9 and Section~6.1]{Gouvea}.
In particular, if $\size{x - 1}_p < p^{-1/(p - 1)}$, then $x = \exp_p \log_p x$.

To interpolate $F(n)$, we use this last identity to rewrite $\phi^n$ and $\bar\phi^n$.
There are several ways to do this.
We use an approach that directly involves the root of unity $\omega(\phi)$, which plays a major role in the structure of the set $F(\Z_p)$.
Since $e \geq 1$ and $p \neq 2$, we have $\big\lvert \frac{\phi}{\omega(\phi)} - 1 \big\lvert_p < p^{-1/(p - 1)}$; this also follows from a more general result~\cite[Lemma~6]{Rowland--Yassawi}.
Therefore, for all $n \geq 0$, we have $\big(\frac{\phi}{\omega(\phi)}\big)^n = \exp_p \log_p\!\big(\big(\frac{\phi}{\omega(\phi)}\big)^n\big) = \exp_p\!\big(n \log_p \frac{\phi}{\omega(\phi)}\big)$.

\begin{theorem}[{Rowland--Yassawi~\cite[Theorem~15]{Rowland--Yassawi}}]\label{interpolation}
Let $p \neq 2$ be a prime, and let $p^f$ be the size of the residue field $\mathcal O_K/\{x \in K : |x|_p < 1\}$.
For each $i \in \{0, 1, \dots, p^f - 2\}$, define the function $F_i \colon \Z_p \to K$ by
\[
	F_i(x)
	= \frac{\omega(\phi)^i \exp_p\!\left(x \log_p \tfrac{\phi}{\omega(\phi)}\right) - \omega(\bar\phi)^i \exp_p\!\left(-x \log_p \tfrac{\phi}{\omega(\phi)}\right)}{\sqrt{5}}.
\]
Then $F_i(\Z_p) \subseteq \Z_p$, and $F(n) = F_{n \bmod (p^f - 1)}(n)$ for all $n \geq 0$.
\end{theorem}

To see that $F_i(x) \in \Z_p$ if $x \in \Z_p$, take a sequence of integers $(x_m)_{m \geq 0}$ that converges to $x$ such that $x_m \equiv i \mod p^f - 1$ for each $m \geq 0$; since $F_i(x_m) = F(x_m) \in \Z$, it follows by continuity that $F_i(x) \in \Z_p$.

Note we have not defined $F_i$ for $p = 2$; the definition requires modification and will be discussed in Section~\ref{p = 2}.

Recall the observation in the proof of Proposition~\ref{Lucas zeros} that $\alpha(p)$ is the order of $-\phi^2$ modulo~$p$.
Along with Theorem~\ref{ratio}, this implies that $\omega(\phi)^{\pi(p)} = 1$ and $\omega(\bar\phi)^{\pi(p)} = 1$; Ballot and Elia give an alternate proof of this fact~\cite[Proposition~1.2]{Ballot--Elia}.
Therefore $F_{i + \pi(p)}(x) = F_i(x)$ for all $x \in \Z_p$, so we can reduce the number of functions in the piecewise interpolation.
Namely, the functions $F_i(x)$ for $i \in \{0, 1, \dots, \pi(p) - 1\}$ comprise an interpolation.

In Section~\ref{Proof of the main result}, we will use Proposition~\ref{isomorphism} below, which refines the isomorphism $\log_p$.
To prove it, we need two lemmas, whose proofs we include for completeness.
The first is the lifting-the-exponent lemma.

\begin{lemma}\label{powers}
Let $e \geq 0$ be an integer.
If $a, b \in \mathcal O_K$ such that $\lvert a - b \rvert_p < p^{-1/(p - 1)}$, then $\left\lvert a^{p^e} - b^{p^e} \right\rvert_p < p^{-e - 1/(p - 1)}$.
\end{lemma}

\begin{proof}
Assume $i \geq 0$ and $\lvert a - b \rvert_p < p^{-i - 1/(p - 1)}$; we show that this implies $\lvert a^p - b^p \rvert_p < p^{-(i + 1) - 1/(p - 1)}$.
The statement will then follow inductively.
Write $a = b + O(p^i)$.
We have
\begin{align*}
	\frac{a^p - b^p}{a - b}
	&= \sum_{j = 0}^{p - 1} a^{p - 1 - j} b^j
	= \sum_{j = 0}^{p - 1} \left(b + O(p^i)\right)^{p - 1 - j} b^j \\
	&= \sum_{j = 0}^{p - 1} \left(b^{p - 1} + O(p^i)\right)
	= p b^{p - 1} + O(p^i).
\end{align*}
Therefore
\begin{align*}
	\left\lvert a^p - b^p \right\rvert_p
	&= \lvert a - b \rvert_p \cdot \left\lvert p b^{p - 1} + O(p^i) \right\rvert_p \\
	&< p^{-i - 1/(p - 1)} \cdot p^{-1}.
	\qedhere
\end{align*}
\end{proof}

The second is a special case of the isometry property of the $p$-adic logarithm.

\begin{lemma}\label{log size}
If $x \in \mathcal O_K$ and $\lvert x \rvert_p < p^{-1/(p - 1)}$, then $\left\lvert\log_p(1 + x)\right\rvert_p = \left\lvert x \right\rvert_p$.
\end{lemma}

\begin{proof}
The power series for $\log_p$ gives
\[
	\left\lvert\log_p(1 + x)\right\rvert_p
	= \lvert x \rvert_p \cdot \left\lvert 1 + \sum_{m \geq 2} (-1)^{m + 1} \frac{x^{m - 1}}{m} \right\rvert_p.
\]
Since $\lvert m \rvert_p \geq p^{-\frac{m - 1}{p - 1}}$ for each $m \geq 1$, for $m \geq 2$ we have $\left\lvert \frac{x^{m - 1}}{m} \right\rvert_p < \frac{1}{\lvert m \rvert_p} p^{-\frac{m - 1}{p - 1}} \leq 1$.
Therefore
\[
	\left\lvert 1 + \sum_{m \geq 2} (-1)^{m + 1} \frac{x^{m - 1}}{m} \right\rvert_p
	= 1
\]
by the ultrametric inequality.
\end{proof}

\begin{proposition}\label{isomorphism}
For each integer $e \geq 0$, $\exp_p$ is a group isomorphism from $\{x : \lvert x \rvert_p < p^{-e - 1/(p - 1)}\}$ to $\{x : \lvert x - 1 \rvert_p < p^{-e - 1/(p - 1)}\}$.
\end{proposition}

\begin{proof}
Since $\exp_p$ is a group isomorphism from $\{x : \lvert x \rvert_p < p^{-1/(p - 1)}\}$ to $\{x : \lvert x - 1 \rvert_p < p^{-1/(p - 1)}\}$, to prove the proposition it suffices to show that
\[
	\left\{\exp_p(x) : \lvert x \rvert_p < p^{-e - 1/(p - 1)}\right\}
	= \left\{x : \lvert x - 1 \rvert_p < p^{-e - 1/(p - 1)}\right\}.
\]

To show $\subseteq$, let $x \in \mathcal O_K$ such that $\lvert x \rvert_p < p^{-e - 1/(p - 1)}$; then $\lvert \frac{x}{p^e} \rvert_p < p^{-1/(p - 1)}$, so $\frac{x}{p^e}$ is in the domain of $\exp_p$.
We have
\[
	\exp_p(x)
	= \exp_p(\underbrace{\tfrac{x}{p^e} + \tfrac{x}{p^e} + \dots + \tfrac{x}{p^e}}_{p^e})
	= \left(\exp_p(\tfrac{x}{p^e})\right)^{p^e}.
\]
Since $\lvert \exp_p(\frac{x}{p^e}) - 1 \rvert_p < p^{-1/(p - 1)}$, Lemma~\ref{powers} implies $\big\lvert (\exp_p(\frac{x}{p^e}))^{p^e} - 1 \big\rvert_p < p^{-e - 1/(p - 1)}$.

To show $\supseteq$, let $y \in \mathcal O_K$ such that $\lvert y - 1 \rvert_p < p^{-e - 1/(p - 1)}$.
Let $x = \log_p y$.
By Lemma~\ref{log size}, $\lvert x \rvert_p < p^{-e - 1/(p - 1)}$, and $x$ satisfies $\exp_p(x) = y$.
\end{proof}

\section{Proof of the main result}\label{Proof of the main result}

In this section we prove Theorem~\ref{main theorem}, which establishes the value of $\dens(p)$ for $p \neq 2$.
This density is equal to the measure of the closure of the set of Fibonacci numbers in the $p$-adic integers $\Z_p$.
That is, let $\mu$ be the Haar measure on $\Z_p$ defined by $\mu(m + p^\lambda \Z_p) = p^{-\lambda}$;
then $\dens(p) = \mu(\overline{F(\N)})$.

As an illustrative warm-up example, we determine the limiting density of residues attained by the set of squares modulo powers of $p$ for $p \neq 2$.
Let $z \in \Z_p$, and define $\lambda \in \Z$ by $\size{z}_p = \frac{1}{p^\lambda}$.
If $\lambda$ is odd, then $z$ does not have a square root in $\Z_p$.
If $\lambda$ is even, then $z$ has a square root in $\Z_p$ (more precisely, in $p^{\lambda/2} \Z_p$) if and only if $\frac{z}{p^\lambda} \bmod p$ is a quadratic residue, by Hensel's lemma.
Since there are $\frac{p - 1}{2}$ nonzero quadratic residues modulo~$p$, the image of $\Z_p$ under $x \mapsto x^2$ has measure
\[
	\sum_{\substack{\lambda \geq 0 \\ \text{$\lambda$ even}}} \frac{p - 1}{2} \cdot \frac{1}{p^{\lambda + 1}}
	= \frac{p}{2 (p + 1)}.
\]
The term $\frac{Z(p)}{2 p^{2 e - 1} (p + 1)}$ in Theorem~\ref{main theorem} will arise in a similar way, since the Fibonacci sequence satisfies a second-order recurrence.

\begin{notation*}
For $p \neq 2$, we write the function $F_i(x)$ in Theorem~\ref{interpolation} as the composition $F_i(x) = g_i(h_i(x))$ where $g_i(y) = \frac{y - (-1)^i y^{-1}}{\sqrt{5}}$ and $h_i(x) = \omega(\phi)^i \exp_p\!\left(x \log_p \frac{\phi}{\omega(\phi)}\right)$.
We use this notation throughout this section.
The function $g_i$ is close enough to a polynomial that we will be able to apply Hensel's lemma.
\end{notation*}

To prove Theorem~\ref{main theorem}, we describe the set $F_i(\Z_p)$.
In Proposition~\ref{subset} we show that this set is contained in $F(i) + p^e \Z_p$, where $e$ is defined by $\size{\frac{\phi}{\omega(\phi)} - 1}_p = \frac{1}{p^e}$.
For Lucas non-zeros $i$, Proposition~\ref{Lucas non-zero image} shows that in fact $F_i(\Z_p) = F(i) + p^e \Z_p$.
It follows that the Lucas non-zeros contribute measure $\frac{N(p)}{p^e}$ in Theorem~\ref{main theorem}.
Moreover, they account for the full subtrees rooted at level $1$ in Figure~\ref{p=7 tree} and analogous trees for other primes.

When $i$ is a Lucas zero, the description of the set $F_i(\Z_p)$ is more complicated, since it is not a cylinder set.
We will show that partial branching of the kind pictured in Figure~\ref{p=7 tree} occurs through the tree along edges corresponding to the $p$-adic digits of $\omega(\phi)^i \frac{2}{\sqrt{5}}$.
For $p = 7$, there are two such paths, for the $7$-adic numbers
\begin{align*}
	\omega(\phi)^4 \tfrac{2}{\sqrt{5}} &= 3 + \frac{0}{7^{-1}} + \frac{3}{7^{-2}} + \frac{4}{7^{-3}} + \frac{3}{7^{-4}} + \frac{6}{7^{-5}} + \frac{4}{7^{-6}} + \frac{1}{7^{-7}} + \cdots \\
	\omega(\phi)^{12} \tfrac{2}{\sqrt{5}} &= 4 + \frac{6}{7^{-1}} + \frac{3}{7^{-2}} + \frac{2}{7^{-3}} + \frac{3}{7^{-4}} + \frac{0}{7^{-5}} + \frac{2}{7^{-6}} + \frac{5}{7^{-7}} + \cdots
\end{align*}
(where we place each power of $7$ in the denominator to make the digits easier to read).
This partial branching is due to a double root of $y^2 - \sqrt{5} z y - (-1)^i = 0$ (equivalently, $g_i(y) = z$) when $z = \omega(\phi)^i \frac{2}{\sqrt{5}}$.
By computing digits of $\omega(\phi)^i \frac{2}{\sqrt{5}}$, we can use the set of quadratic residues modulo~$p$ to quickly construct levels $\lambda \geq 2$ of the tree of attained residues.
The construction follows from Lemma~\ref{Lucas zero image}, and Proposition~\ref{Lucas zero measure} establishes $\mu(F_i(\Z_p))$ for Lucas zeros $i$.

We begin the proof of Theorem~\ref{main theorem} with a lemma allowing us to conclude that certain numbers belong to $\exp_p(p \sqrt{5} \Z_p)$.
As before, let $\bar\phi = \frac{1 - \sqrt 5}{2}$.
In the case $p \equiv 2, 3 \mod 5$, we generalize this notation as follows.
If $x = a + b \sqrt{5}$ for some $a, b \in \Q_p$, define $\bar{x} = a - b \sqrt{5}$.

\begin{lemma}\label{pure square root 5}
Let $p \equiv 2, 3 \mod 5$.
If $x \in \mathcal O_K$ such that $x \bar{x} = 1$ and $\size{x - 1}_p < p^{-1/(p - 1)}$, then $\log_p x \in p \sqrt{5} \Z_p$.
\end{lemma}

\begin{proof}
In addition to $\size{x - 1}_p < p^{-1/(p - 1)}$, we have $\size{\bar{x} - 1}_p = \left\lvert\frac{1}{x} - 1\right\rvert_p = \size{1 - x}_p < p^{-1/(p - 1)}$.
Since $\log_p$ is a group isomorphism from $\{x : \lvert x - 1 \rvert_p < p^{-1/(p - 1)}\}$ to $\{x : \lvert x \rvert_p < p^{-1/(p - 1)}\}$, it follows that $\log_p x + \log_p \bar{x} = \log_p(x \bar{x}) = \log_p 1 = 0$.
Write $\log_p x = a + b \sqrt{5}$ where $a, b \in \Z_p$.
Then
\[
	a - b \sqrt{5}
	= \overline{a + b \sqrt{5}}
	= \log_p \bar{x}
	= -\log_p x
	= -a - b \sqrt{5},
\]
so $a = 0$.
Finally, $b \equiv 0 \mod p$ because $\left\lvert\log_p x\right\rvert_p = \size{x - 1}_p < p^{-1/(p - 1)}$ by Lemma~\ref{log size}.
\end{proof}

Next we use Lemma~\ref{pure square root 5} to describe the image of $\Z_p$ under $h_i$.

\begin{lemma}\label{image under h}
Let $p$ be a prime such that $p \neq 2$ and $p \neq 5$, and define $e \geq 1$ by $\big\lvert\frac{\phi}{\omega(\phi)} - 1\big\rvert_p = \frac{1}{p^e}$.
If $i \in \{0, 1, \dots, \pi(p) - 1\}$, then $h_i(\Z_p) = \omega(\phi)^i \exp_p(p^e \sqrt{5} \Z_p)$.
\end{lemma}

\begin{proof}
It follows from Lemma~\ref{log size} that $\left\lvert\log_p \frac{\phi}{\omega(\phi)}\right\rvert_p = \frac{1}{p^e}$.
In particular, $x \log_p \frac{\phi}{\omega(\phi)}$ is in the domain of $\exp_p$ for all $x \in \Z_p$.
If $p \equiv 1, 4 \mod 5$, then
\begin{align*}
	h_i(\Z_p)
	&= \omega(\phi)^i \exp_p\!\left(\Z_p \log_p \tfrac{\phi}{\omega(\phi)}\right) \\
	&= \omega(\phi)^i \exp_p(p^e \Z_p) \\
	&= \omega(\phi)^i \exp_p(p^e \sqrt{5} \Z_p)
\end{align*}
since $\size{\sqrt{5}}_p = 1$.
If $p \equiv 2, 3 \mod 5$, then $\log_p \frac{\phi}{\omega(\phi)} \in p^e \sqrt{5} \Z_p$ by Lemma~\ref{pure square root 5}, since $\frac{\phi}{\omega(\phi)} \cdot \frac{\bar\phi}{\omega(\bar\phi)} = \frac{-1}{-1} = 1$.
Therefore
\begin{align*}
	h_i(\Z_p)
	&= \omega(\phi)^i \exp_p\!\left(\Z_p \log_p \tfrac{\phi}{\omega(\phi)}\right) \\
	&= \omega(\phi)^i \exp_p(p^e \sqrt{5} \Z_p).
	\qedhere
\end{align*}
\end{proof}

We will not describe the set $h_i(\Z_p) = \omega(\phi)^i \exp_p(p^e \sqrt{5} \Z_p)$ more explicitly.
Instead, the next lemma gives conditions under which $g_i(y) = z$ has a solution $y \in h_i(\Z_p)$ when $p \equiv 2, 3 \mod 5$.
We will apply it for Lucas non-zeros $i$ as well as Lucas zeros.

\begin{lemma}\label{y set general}
Let $p \equiv 2, 3 \mod 5$.
If $p \neq 2$, define $e \geq 1$ by $\big\lvert\frac{\phi}{\omega(\phi)} - 1\big\rvert_p = \frac{1}{p^e}$; if $p = 2$, let $e = 2$.
Let $z, w \in \Z_p$ and $y_0 \in \mathcal O_K \setminus \{0\}$ such that $w^2 = 5 z^2 + 4 y_0 \overline{y_0}$ and $\big\lvert \frac{w + \sqrt{5} z}{2} - y_0 \big\rvert_p \leq \frac{1}{p^e}$.
Then $\frac{w + \sqrt{5} z}{2} \in y_0 \exp_p(p^e \sqrt{5} \Z_p)$.
\end{lemma}

\begin{proof}
We have
\[
	\frac{
		\frac{w + \sqrt{5} z}{2}
	}{
		y_0
	}
	\cdot
	\frac{
		\frac{w - \sqrt{5} z}{2}
	}{
		\overline{y_0}
	}
	= 1.
\]
Since $\big\lvert \frac{w + \sqrt{5} z}{2} - y_0 \big\rvert_p \leq p^{-e} < p^{-1/(p - 1)}$, Lemma~\ref{pure square root 5} implies $\log_p \frac{w + \sqrt{5} z}{2 y_0} \in p \sqrt{5} \Z_p$.
Moreover, Lemma~\ref{log size} implies $\big\lvert \log_p \frac{w + \sqrt{5} z}{2 y_0} \big\rvert_p = \big\lvert \frac{w + \sqrt{5} z}{2} - y_0 \big\rvert_p \leq \frac{1}{p^e}$, so in fact $\log_p \frac{w + \sqrt{5} z}{2 y_0} \in p^e \sqrt{5} \Z_p$, and this implies $\frac{w + \sqrt{5} z}{2} \in y_0 \exp_p(p^e \sqrt{5} \Z_p)$.
\end{proof}

Now we begin to consider $F_i(\Z_p)$.
We will use the following as a partial converse of Lemma~\ref{y set general} to prove Propositions~\ref{subset} and \ref{Lucas zero measure} concerning $F_i(\Z_p)$.

\begin{lemma}\label{y distance}
Let $p$ be a prime such that $p \neq 2$ and $p \neq 5$, and define $e \geq 1$ by $\big\lvert\frac{\phi}{\omega(\phi)} - 1\big\rvert_p = \frac{1}{p^e}$.
Let $i \in \{0, 1, \dots, \pi(p) - 1\}$.
If $y \in h_i(\Z_p)$, then $\lvert y - \omega(\phi)^i \rvert_p \leq \frac{1}{p^e}$.
\end{lemma}

\begin{proof}
Since $F_i(x) = g_i(h_i(x))$, it suffices to show $g_i(y) \in F(i) + p^e \Z_p$.
By Lemma~\ref{image under h}, $y = \omega(\phi)^i \exp_p(p^e \sqrt{5} t)$ for some $t \in \Z_p$.
Since $p \neq 2$, we have $1 > \frac{1}{p - 1}$ and
\[
	\big\lvert p^e \sqrt{5} t \big\rvert_p
	\leq p^{-e}
	< p^{-e + 1 - 1/(p - 1)}.
\]
By Proposition~\ref{isomorphism}, this implies
\[
	\left\lvert y - \omega(\phi)^i \right\rvert_p
	= \left\lvert \exp_p(p^e \sqrt{5} t) - 1 \right\rvert_p
	< p^{-e + 1 - 1/(p - 1)}.
\]
Since $p \neq 5$, we have $\lvert y - \omega(\phi)^i \rvert_p \in p^\Z$ (because the ramification index is $1$); therefore $\lvert y - \omega(\phi)^i \rvert_p \leq p^{-e}$.
\end{proof}

\begin{proposition}\label{subset}
Let $p$ be a prime such that $p \neq 2$ and $p \neq 5$, and define $e \geq 1$ by $\big\lvert\frac{\phi}{\omega(\phi)} - 1\big\rvert_p = \frac{1}{p^e}$.
If $i \in \{0, 1, \dots, \pi(p) - 1\}$, then $F_i(\Z_p) \subseteq F(i) + p^e \Z_p$.
\end{proposition}

\begin{proof}
Let $y \in h_i(\Z_p)$.
By Lemma~\ref{y distance}, $\lvert y - \omega(\phi)^i \rvert_p \leq p^{-e}$.
We also have $\lvert\phi^i - \omega(\phi)^i\rvert_p \leq p^{-e}$ since $\big\lvert\frac{\phi}{\omega(\phi)} - 1\big\rvert_p = p^{-e}$.
Therefore
\begin{align*}
	\left\lvert g_i(y) - F(i) \right\rvert_p
	&= \left\lvert \frac{y - (-1)^i y^{-1}}{\sqrt{5}} - \frac{\phi^i - (-1)^i \phi^{-i}}{\sqrt{5}} \right\rvert_p \\
	&= \left\lvert \left(y - \phi^i\right) - (-1)^i \left(y^{-1} - \phi^{-i}\right) \right\rvert_p \\
	&\leq \max\!\left(
		\left\lvert y - \phi^i \right\rvert_p,
		\left\lvert y^{-1} - \phi^{-i} \right\rvert_p
	\right) \\
	&= \left\lvert y - \phi^i \right\rvert_p \\
	&\leq \max\!\left(
		\left\lvert y - \omega(\phi)^i \right\rvert_p,
		\left\lvert \omega(\phi)^i - \phi^i \right\rvert_p
	\right) \\
	&\leq p^{-e}
\end{align*}
after two applications of the ultrametric inequality.
It follows that $g_i(y) \in F(i) + p^e \mathcal O_K$.
By Theorem~\ref{interpolation}, $g_i(y) \in F_i(\Z_p) \subseteq \Z_p$, so $g_i(y) \in F(i) + p^e \Z_p$.
\end{proof}

If $i$ is a Lucas non-zero, the $\subseteq$ in Proposition~\ref{subset} can be strengthened to $=$ as follows; this establishes that $\mu(F_i(\Z_p)) = \frac{1}{p^e}$.

\begin{proposition}\label{Lucas non-zero image}
Let $p$ be a prime such that $p \neq 2$ and $p \neq 5$, and define $e \geq 1$ by $\big\lvert\frac{\phi}{\omega(\phi)} - 1\big\rvert_p = \frac{1}{p^e}$.
If $i$ is a Lucas non-zero, then $F_i(\Z_p) = F(i) + p^e \Z_p$.
\end{proposition}

\begin{proof}
As a result of Proposition~\ref{subset}, it suffices to show
\[
	g_i\!\left(\omega(\phi)^i \exp_p(p^e \sqrt{5} \Z_p)\right) \supseteq F(i) + p^e \Z_p.
\]
Let $z \in F(i) + p^e \Z_p$.
We first show that there exists $y \in \omega(\phi)^i + p^e \mathcal O_K$ such that $g_i(y) = z$, and then we show $y \in \omega(\phi)^i \exp_p(p^e \sqrt{5} \Z_p)$.

The equation $g_i(y) = z$ is equivalent to $y^2 - \sqrt{5} z y - (-1)^i = 0$.
Let $y_0 = \omega(\phi)^i$.
Binet's formula~\eqref{Binet} shows that $\frac{y_0 - (-1)^i y_0^{-1}}{\sqrt{5}} \equiv \frac{\phi^i - (-1)^i \phi^{-i}}{\sqrt{5}} = F(i) \equiv z \mod p^e$, so $y_0^2 - \sqrt{5} z y_0 - (-1)^i \equiv 0 \mod p^e$.
The discriminant of $y^2 - \sqrt{5} z y - (-1)^i$ is $5 z^2 + (-1)^i 4 \equiv 5 F(i)^2 + (-1)^i 4 = L(i)^2 \mod p^e$.
Therefore $(5 z^2 + (-1)^i 4) \bmod p$ is a quadratic residue.
Since $i$ is a Lucas non-zero, we have $5 z^2 + (-1)^i 4 \nequiv 0 \mod p$, so $\Z_p$ contains two distinct square roots of $5 z^2 + (-1)^i 4$.
The two solutions to $y^2 - \sqrt{5} z y - (-1)^i = 0$ are $y = \frac{\pm w + \sqrt{5} z}{2}$, where $w \in \Z_p$ is a square root of $5 z^2 + (-1)^i 4$.
Choose $w$ such that $\frac{w + \sqrt{5} z}{2} \in \omega(\phi)^i + p^e \mathcal O_K$.
It remains to show that $\frac{w + \sqrt{5} z}{2} \in \omega(\phi)^i \exp_p(p^e \sqrt{5} \Z_p)$.

If $p \equiv 1, 4 \mod 5$, then $\sqrt{5} \in \Z_p$ and Proposition~\ref{isomorphism} implies that $\exp_p$ is an isomorphism from $p^e \Z_p$ to $1 + p^e \Z_p$, so
\[
	\omega(\phi)^i \exp_p(p^e \sqrt{5} \Z_p)
	= \omega(\phi)^i(1 + p^e \Z_p)
	= \omega(\phi)^i + p^e \Z_p.
\]
Therefore $\frac{w + \sqrt{5} z}{2} \in \omega(\phi)^i \exp_p(p^e \sqrt{5} \Z_p)$.

Let $p \equiv 2, 3 \mod 5$.
We have $y_0 \overline{y_0} = \omega(\phi)^i \omega(\bar\phi)^i = (-1)^i$, so $w^2 = 5 z^2 + 4 y_0 \overline{y_0}$.
By Lemma~\ref{y set general}, $\frac{w + \sqrt{5} z}{2} \in \omega(\phi)^i \exp_p(p^e \sqrt{5} \Z_p)$.
\end{proof}

The last major step in the proof of Theorem~\ref{main theorem} is to establish the measure of $F_i(\Z_p)$ for Lucas zeros $i$.
We do this in Proposition~\ref{Lucas zero measure}, using the following two lemmas.

\begin{lemma}\label{part 1}
Let $p$ be a prime such that $p \neq 2$ and $p \neq 5$.
Let $i$ be a Lucas zero, and define $\zeta = \omega(\phi)^i$.
Then $\zeta \in \sqrt{5} \Z_p$.
\end{lemma}

\begin{proof}
Since $i$ is a Lucas zero and $L(i) = \frac{\phi^{2 i} + (-1)^i}{\phi^i}$, we have $\zeta^2 = \omega(\phi^{2 i}) = - (-1)^i$.
(In particular, $\zeta^4 = 1$.)

If $p \equiv 1, 4 \mod 5$, then $\sqrt{5} \in \Z_p$, so it follows that $\zeta \in \sqrt{5} \Z_p$ as desired.
If $p \equiv 2, 3 \mod 5$, write $\zeta = a + b \sqrt{5}$ where $a, b \in \Z_p$.
Since $a^2 + 2 a b \sqrt{5} + 5 b^2 = \zeta^2 \in \Z_p$, this implies $a = 0$ or $b = 0$, so $\zeta \in \Z_p$ or $\zeta \in \sqrt{5} \Z_p$.
Vinson~\cite[Theorem~4]{Vinson} showed that $\frac{\pi(p)}{\alpha(p)} = 2$ if $p \equiv 3, 7 \mod 20$ and $\frac{\pi(p)}{\alpha(p)} = 4$ if $p \equiv 13, 17 \mod 20$.
These two cases include all primes $p \equiv 2, 3 \mod 5$ such that $p \neq 2$, so Theorem~\ref{ratio} and Proposition~\ref{Lucas zeros} imply that $i$ is even.
Therefore $\zeta^2 = - (-1)^i = -1$.
If $p \equiv 3 \mod 4$, then one of the supplements to quadratic reciprocity implies $\zeta \notin \Z_p$, so $\zeta \in \sqrt{5} \Z_p$.
If $p \equiv 1 \mod 4$, then $p \equiv 13, 17 \mod 20$, so $\frac{\pi(p)}{\alpha(p)} = 4$ and there are no Lucas zeros, so the statement is vacuously true.
\end{proof}

Since $\zeta \in \sqrt{5} \Z_p$ by Lemma~\ref{part 1}, for a given $j \in \Z_p$ the residue $\zeta \sqrt{5} j \bmod p$ is either a quadratic residue or a quadratic non-residue.
This allows us, for a given $z$, to characterize the Lucas zeros $i$ for which $y^2 - \sqrt{5} z y - (-1)^i = 0$ has a solution $y \in \Z_p$.

\begin{lemma}\label{Lucas zero image}
Let $p$ be a prime such that $p \neq 2$ and $p \neq 5$.
Let $i$ be a Lucas zero, and define $\zeta = \omega(\phi)^i$.
Let $z \in \Z_p$.
Define $\lambda \in \Z$ by $\big\lvert z - \zeta \frac{2}{\sqrt{5}} \big\rvert_p = \frac{1}{p^\lambda}$, and define $j \in \{1, 2, \dots, p - 1\}$ by $z \equiv \zeta \frac{2}{\sqrt{5}} + j p^\lambda \mod p^{\lambda + 1}$.
\begin{enumerate}
\item\label{part 2}
The equation $y^2 - \sqrt{5} z y - (-1)^i = 0$ has a solution $y \in \Z_p$ if and only if $\lambda$ is even and $\zeta \sqrt{5} j \bmod p$ is a quadratic residue.
\item\label{part 3}
If $y^2 - \sqrt{5} z y - (-1)^i = 0$ has a solution $y \in \Z_p$, then $\lvert y - \zeta \rvert_p = \frac{1}{p^{\lambda/2}}$.
\end{enumerate}
\end{lemma}

To prove Lemma~\ref{Lucas zero image}, we use the following version of Hensel's lemma in $\mathcal O_K$.

\begin{Hensel}
Let $f(x) \in \mathcal O_K[x]$.
If there exists $y_0 \in \mathcal O_K$ such that $\lvert f(y_0) \rvert_p < \lvert f'(y_0) \rvert_p^2$, then there is a unique $y \in \mathcal O_K$ satisfying $\lvert y - y_0 \rvert_p < \lvert f'(y_0) \rvert_p$ such that $f(y) = 0$.
\end{Hensel}

\begin{proof}[Proof of Lemma~\ref{Lucas zero image}]
Let $f_z(y) \colonequal y^2 - \sqrt{5} z y - (-1)^i$.
As in the proof of Lemma~\ref{part 1}, $\zeta^2 = \omega(\phi^{2 i}) = - (-1)^i$.
Therefore we can write $f_z(y) = y^2 - \sqrt{5} z y + \zeta^2$.

To prove Assertion~\eqref{part 2}, first assume $\lambda$ is even and $\zeta \sqrt{5} j \bmod p$ is a quadratic residue.
Let $a \in \{1, \dots, p - 1\}$ such that $a^2 \equiv \zeta \sqrt{5} j \mod p$.
We now check that $y_0 \colonequal \zeta + p^{\lambda/2} a$ satisfies the condition of Hensel's lemma.
Using $\sqrt{5} z \equiv 2 \zeta + \sqrt{5} j p^\lambda \mod p^{\lambda + 1}$, we have
\begin{align*}
	f_z(y_0)
	&= y_0^2 - \sqrt{5} z y_0 + \zeta^2 \\
	&\equiv \left(\zeta + p^{\lambda/2} a\right)^2 - \left(2 \zeta + \sqrt{5} j p^\lambda\right) \left(\zeta + p^{\lambda/2} a\right) + \zeta^2 \mod p^{\lambda + 1} \\
	&\equiv \left(a^2 - \zeta \sqrt{5} j\right) p^{\lambda} \mod p^{\lambda + 1} \\
	&\equiv 0 \mod p^{\lambda + 1}.
\end{align*}
Therefore $\lvert f_z(y_0) \rvert_p \leq \frac{1}{p^{\lambda + 1}}$.
It remains to bound $\lvert f_z'(y_0) \rvert_p^2$.
Since $z \equiv \zeta \frac{2}{\sqrt{5}} \mod p^{\lambda/2}$, we have $f_z'(y_0) \equiv 2 y_0 - 2 \zeta \equiv 0 \mod p^{\lambda/2}$, but $f_z'(y_0) \nequiv 0 \mod p^{\lambda/2 + 1}$ since $a \nequiv 0 \mod p$.
Therefore
\[
	\lvert f_z(y_0) \rvert_p
	\leq p^{-\lambda - 1}
	< p^{-\lambda}
	= \lvert f_z'(y_0) \rvert_p^2,
\]
so Hensel's lemma implies that $f_z(y)$ has a unique root $y \in \Z_p$ satisfying $\lvert y - (\zeta + p^{\lambda/2} a) \rvert_p < \frac{1}{p^{\lambda/2}}$.
It follows that $\lvert y - \zeta \rvert_p = \frac{1}{p^{\lambda/2}}$.
This proves one direction of Assertion~\eqref{part 2} and also Assertion~\eqref{part 3}.

Conversely, assume $y^2 - \sqrt{5} z y + \zeta^2 = 0$ and $y \in \Z_p$.
Since $z \equiv \zeta \frac{2}{\sqrt{5}} \mod p^\lambda$, we have
\begin{align*}
	0
	&= y^2 - \sqrt{5} z y + \zeta^2 \\
	&\equiv y^2 - 2 \zeta y + \zeta^2 \mod p^\lambda \\
	&= (y - \zeta)^2.
\end{align*}
There are two cases.

If $\lambda$ is even, this implies $y \equiv \zeta \mod p^{\lambda/2}$.
Write $y = \zeta + p^{\lambda/2} a$ for some $a \in \Z_p$.
Expanding $f_z(y)$ as above shows that $0 = f_z(y) \equiv \left(a^2 - \zeta \sqrt{5} j\right) p^{\lambda} \mod p^{\lambda + 1}$.
Therefore $\zeta \sqrt{5} j \bmod p$ is a quadratic residue.

If $\lambda$ is odd, then $0 \equiv (y - \zeta)^2 \mod p^\lambda$ implies $y \equiv \zeta \mod p^{(\lambda + 1)/2}$.
Write $y = \zeta + p^{(\lambda + 1)/2} a$ for some $a \in \Z_p$.
Then
\begin{align*}
	0
	&= y^2 - \sqrt{5} z y + \zeta^2 \\
	&\equiv \left(\zeta + p^{(\lambda + 1)/2} a\right)^2 - \left(2 \zeta + \sqrt{5} j p^\lambda\right) \left(\zeta + p^{(\lambda + 1)/2} a\right) + \zeta^2 \mod p^{\lambda + 1} \\
	&\equiv -\zeta \sqrt{5} j p^{\lambda} \mod p^{\lambda + 1},
\end{align*}
which contradicts $j \in \{1, 2, \dots, p - 1\}$.
Therefore there is no solution $y \in \Z_p$ when $\lambda$ is odd.
\end{proof}

\begin{proposition}\label{Lucas zero measure}
Let $p$ be a prime such that $p \neq 2$ and $p \neq 5$, and define $e \geq 1$ by $\big\lvert\frac{\phi}{\omega(\phi)} - 1\big\rvert_p = \frac{1}{p^e}$.
If $i$ is a Lucas zero, then $\mu(F_i(\Z_p)) = \frac{1}{2 p^{2 e - 1} (p + 1)}$.
\end{proposition}

\begin{proof}
Let $i$ be a Lucas zero.
We would like to determine the measure of the set
\[
	F_i(\Z_p) = \{z \in \Z_p : \text{$g_i(y) = z$ for some $y \in h_i(\Z_p)$}\}.
\]
Let $z \in \Z_p$.
As in the proof of Proposition~\ref{Lucas non-zero image}, $g_i(y) = z$ if and only if $y^2 - \sqrt{5} z y - (-1)^i = 0$.
Let $\zeta = \omega(\phi)^i$.
Define $\lambda \in \Z$ by $\big\lvert z - \zeta \frac{2}{\sqrt{5}} \big\rvert_p = \frac{1}{p^\lambda}$, and define $j \in \{1, 2, \dots, p - 1\}$ by $z \equiv \zeta \frac{2}{\sqrt{5}} + j p^\lambda \mod p^{\lambda + 1}$.
By Lemma~\ref{Lucas zero image}, if $\lambda$ is odd then $g_i(y) = z$ has no solution $y \in \Z_p$.
Furthermore, if $\lambda$ is even then $g_i(y) = z$ has a unique solution $y \in \Z_p$ if $\zeta \sqrt{5} j \bmod p$ is a quadratic residue and no solution otherwise.
In the case that there is a unique solution, we will show that $y \in h_i(\Z_p)$ if and only if $\lambda \geq 2 e$.
Since there are $\frac{p - 1}{2}$ nonzero quadratic residues modulo~$p$ and the set $\zeta \frac{2}{\sqrt{5}} + j p^\lambda + p^{\lambda + 1} \Z_p$ has measure $\frac{1}{p^{\lambda + 1}}$, it will then follow that
\[
	\mu(F_i(\Z_p))
	= \sum_{\substack{\lambda \geq 2 e \\ \text{$\lambda$ even}}} \frac{p - 1}{2} \cdot \frac{1}{p^{\lambda + 1}}
	= \frac{1}{2 p^{2 e - 1} (p + 1)}.
\]

To that end, assume that $\lambda$ is even and $\zeta \sqrt{5} j \bmod p$ is a quadratic residue.
Let $y \in \Z_p$ be the unique solution of $g_i(y) = z$, which is guaranteed by Lemma~\ref{Lucas zero image}.
In particular, $\lvert y - \omega(\phi)^i \rvert_p = \frac{1}{p^{\lambda/2}}$.

If $\lambda < 2 e$, then $\lvert y - \omega(\phi)^i \rvert_p = \frac{1}{p^{\lambda/2}} > p^{-e}$.
By Lemma~\ref{y distance}, $y \notin h_i(\Z_p)$.

For the other direction, assume $\lambda \geq 2 e$.
Then $\size{y - \omega(\phi)^i}_p = \frac{1}{p^{\lambda/2}} \leq p^{-e}$.
We show that $y \in h_i(\Z_p)$.
There are two cases.

If $p \equiv 1, 4 \mod 5$, then $h_i(\Z_p) = \omega(\phi)^i + p^e \Z_p$ (again by Lemma~\ref{image under h} and the proof of Proposition~\ref{Lucas non-zero image}), so $y \in h_i(\Z_p)$.

Let $p \equiv 2, 3 \mod 5$.
Let $y_0 = \omega(\phi)^i$.
Since $i$ is a Lucas zero, we have $\omega(\phi)^i + (-1)^i \omega(\phi)^{-i} \equiv 0 \mod p$.
In fact, we can strengthen this congruence to $\omega(\phi)^i + (-1)^i \omega(\phi)^{-i} \equiv 0 \mod p^e$ by Theorem~\ref{Wall exponent Lucas characterization} (along with the fact that $e = 1$ for $p = 3$).
Therefore $\frac{y_0 - (-1)^i y_0^{-1}}{\sqrt{5}} = \frac{\omega(\phi)^i - (-1)^i \omega(\phi)^{-i}}{\sqrt{5}} \equiv \zeta \frac{2}{\sqrt{5}} \equiv z \mod p^e$, so $y_0^2 - \sqrt{5} z y_0 - (-1)^i \equiv 0 \mod p^e$.
The discriminant of $y^2 - \sqrt{5} z y - (-1)^i$ is
\begin{align*}
	5 z^2 + (-1)^i 4
	&\equiv 5 \left(\zeta \tfrac{2}{\sqrt{5}} + j p^\lambda\right)^2 + (-1)^i 4 \mod p^{\lambda + 1} \\
	&\equiv 4 \zeta \sqrt{5} j p^\lambda \mod p^{\lambda + 1},
\end{align*}
so $\frac{1}{p^\lambda} (5 z^2 + (-1)^i 4) \equiv 2^2 \cdot \zeta \sqrt{5} j \mod p$ is a quadratic residue.
Let $w \in p^{\lambda/2} \Z_p$ be the square root of $5 z^2 + (-1)^i 4$ satisfying $\frac{w + \sqrt{5} z}{2} \in \omega(\phi)^i + p^e \mathcal O_K$.
The two solutions to $y^2 - \sqrt{5} z y - (-1)^i = 0$ are $y = \frac{\pm w + \sqrt{5} z}{2}$.
By Lemma~\ref{y set general}, $\frac{w + \sqrt{5} z}{2} \in \omega(\phi)^i \exp_p(p^e \sqrt{5} \Z_p)$.
Therefore $\frac{w + \sqrt{5} z}{2} \in h_i(\Z_p)$.
\end{proof}

Putting together the previous results now allows us to prove the main result of the article.

\begin{proof}[Proof of Theorem~\ref{main theorem}]
As mentioned in Section~\ref{Introduction}, for $p = 5$ it follows from Burr's characterization~\cite{Burr} that $\dens(5) = 1$.
We compute $N(5) = 5$ and $Z(5) = 0$, so $\dens(5) = 1 = \frac{N(5)}{5} + \frac{Z(5)}{10 \cdot (5 + 1)}$ as desired.

Let $p$ be a prime such that $p \neq 2$ and $p \neq 5$.
Let $e = \nu_p(F(p - \epsilon))$.
By Theorem~\ref{Wall exponent characterization}, $\big\lvert\frac{\phi}{\omega(\phi)} - 1\big\rvert_p = \frac{1}{p^e}$.

By Theorem~\ref{interpolation} and the discussion following it,
\[
	\dens(p)
	= \mu\!\left(\bigcup_{i = 0}^{p^f - 2} F_i(\Z_p)\right)
	= \mu\!\left(\bigcup_{i = 0}^{\pi(p) - 1} F_i(\Z_p)\right).
\]
By Proposition~\ref{Lucas non-zero image}, if $i$ is a Lucas non-zero, then $F_i(\Z_p) = F(i) + p^e \Z_p$.
The number of distinct images is $N(p)$, so together the Lucas non-zeros contribute measure $\frac{N(p)}{p^e}$.

Let $i$ be a Lucas zero.
By Proposition~\ref{subset}, $F_i(\Z_p) \subseteq F(i) + p^e \Z_p$.
If $F(i) \equiv F(j) \mod p^e$ for some Lucas non-zero $j$, then $F_i(\Z_p) \subseteq F(j) + p^e \Z_p = F_j(\Z_p)$, so $\mu(F_i(\Z_p) \cup F_j(\Z_p)) = \mu(F_j(\Z_p))$ and $F_i(\Z_p)$ makes no additional contribution to $\dens(p)$.
Otherwise, $F_i(\Z_p)$ is disjoint from $F_j(\Z_p)$ for all Lucas non-zeros $j$, so $\mu(F_i(\Z_p) \cup F_j(\Z_p)) = \mu(F_i(\Z_p)) + \mu(F_j(\Z_p))$.
There are $Z(p)$ Lucas zeros $i$ such that $F(i) \nequiv F(j) \mod p^e$ for all Lucas non-zeros $j$.
If $Z(p) \geq 2$ then Proposition~\ref{Lucas zeros} implies that $Z(p) = 2$ and that the two images $F_i(\Z_p)$ are disjoint.
Therefore the Lucas zeros contribute measure $\frac{Z(p)}{2 p^{2 e - 1} (p + 1)}$ by Proposition~\ref{Lucas zero measure}.
\end{proof}

\section{The prime $p = 2$}\label{p = 2}

In this section we prove Theorem~\ref{theorem p=2}, which states that $\dens(2) = \frac{21}{32}$.
For the prime $p = 2$, the period length is $\pi(2) = 3 = \alpha(2)$.
There is one Lucas zero, namely $i = 0$.
It is convenient to write elements of $\Q_2(\sqrt{5})$ as elements of $\Q_2(\phi)$ instead since the denominator of $\frac{1 + \sqrt{5}}{2}$ obscures the fact that $\phi$ is a $2$-adic algebraic integer.
Recall that if $\lvert x \rvert_2 = 1$ then $\omega(x)$ is the $3$rd root of unity satisfying $\omega(x) \equiv x \mod 2$.
The roots of unity congruent modulo $2$ to $\phi$ and $\bar\phi$ are
\begin{align*}
	\omega(\phi) &=
		\left(0 + \frac{1}{2^{-1}} + \frac{0}{2^{-2}} + \frac{0}{2^{-3}} + \cdots\right)
		+ \left(1 + \frac{1}{2^{-1}} + \frac{0}{2^{-2}} + \frac{1}{2^{-3}} + \cdots\right) \phi \\
	\omega(\bar\phi) &=
		\left(1 + \frac{0}{2^{-1}} + \frac{1}{2^{-2}} + \frac{1}{2^{-3}} + \cdots\right)
		+ \left(1 + \frac{0}{2^{-1}} + \frac{1}{2^{-2}} + \frac{0}{2^{-3}} + \cdots\right) \phi.
\end{align*}
Their product is $\omega(\phi) \omega(\bar\phi) = 1$ (since $\rho = \omega(\phi) \omega(\bar\phi)$ is a root of unity satisfying $\rho^3 = 1$ and $\rho \equiv \phi \bar\phi = -1 \mod 2$).

The outline of the proof of Theorem~\ref{theorem p=2} is the same as for other primes, but several details are different.
First we need a version of Theorem~\ref{interpolation}.
We compute $\lvert \frac{\phi}{\omega(\phi)} - 1 \rvert_2 = \frac{1}{2}$ and $\lvert \log_2 \frac{\phi}{\omega(\phi)} \rvert_2 = \frac{1}{2} = p^{-1/(p - 1)}$.
The power series for $\exp_2$ does not converge when evaluated at $\log_2 \frac{\phi}{\omega(\phi)}$; in particular, we cannot replace $\frac{\phi}{\omega(\phi)}$ with $\exp_2 \log_2 \frac{\phi}{\omega(\phi)}$.
Consequently the interpolation is comprised of $6$ functions rather than $3$.

\begin{theorem}\label{interpolation p=2}
For each $i \in \{0, 1, 2\}$ and each $r \in \{0, 1\}$, define the function $F_{i, r} \colon r + 2 \Z_2 \to \Q_2(\sqrt{5})$ by
\[
	F_{i, r}(r + 2 x)
	= \frac{\omega(\phi)^{i - r} \phi^r \exp_2\!\left(x \log_2((\tfrac{\phi}{\omega(\phi)})^2)\right) - \omega(\bar\phi)^{i - r} \bar\phi^r \exp_2\!\left(-x \log_2((\tfrac{\phi}{\omega(\phi)})^2)\right)}{\sqrt{5}}.
\]
Then $F_{i, r}(r + 2 \Z_2) \subseteq \Z_2$, and $F(n) = F_{(n \bmod 3), (n \bmod 2)}(n)$ for all $n \geq 0$.
\end{theorem}

\begin{proof}
Let $p = 2$.
We rewrite $\phi^n$ and $\bar\phi^n$ in Binet's formula~\eqref{Binet} as functions that are defined on $\Z_2$.
We have $\lvert (\frac{\phi}{\omega(\phi)})^2 - 1 \rvert_2 = \frac{1}{4} < p^{-1/(p - 1)}$ by direct computation (or by \cite[Lemma~6]{Rowland--Yassawi}).
Therefore $(\frac{\phi}{\omega(\phi)})^2 = \exp_2 \log_2((\frac{\phi}{\omega(\phi)})^2)$.
For $m \geq 0$ and $r \in \{0, 1\}$, write
\begin{align*}
	\phi^{r + 2 m}
	&= \omega(\phi)^{2 m} \phi^r (\tfrac{\phi}{\omega(\phi)})^{2 m}\\
	&= \omega(\phi)^{2 m} \phi^r \exp_2 \log_2((\tfrac{\phi}{\omega(\phi)})^{2 m}) \\
	&= \omega(\phi)^{2 m} \phi^r \exp_2\!\left(m \log_2((\tfrac{\phi}{\omega(\phi)})^2)\right),
\end{align*}
and similarly for $\bar\phi$.
For all $x \in \Z_2$, define
\[
	F_{i, r}(r + 2 x)
	\colonequal \frac{\omega(\phi)^{i - r} \phi^r \exp_2\!\left(x \log_2((\tfrac{\phi}{\omega(\phi)})^2)\right) - \omega(\bar\phi)^{i - r} \bar\phi^r \exp_2\!\left(x \log_2((\tfrac{\bar\phi}{\omega(\bar\phi)})^2)\right)}{\sqrt{5}}.
\]
Then $F_{i, r}$ is an analytic function on $r + 2 \Z_2$ that agrees with $F$ on $A_{i, r} \colonequal \{n \geq 0 : \text{$n \equiv i \mod 3$ and $n \equiv r \mod 2$}\}$.
Finally, since $\phi \bar\phi = -1$,
\begin{align*}
	\log_2((\tfrac{\phi}{\omega(\phi)})^2) + \log_2((\tfrac{\bar\phi}{\omega(\bar\phi)})^2)
	&= \log_2((\tfrac{\phi}{\omega(\phi)})^2 \cdot (\tfrac{\bar\phi}{\omega(\bar\phi)})^2) \\
	&= \log_2 1 \\
	&= 0,
\end{align*}
so $\log_2((\tfrac{\bar\phi}{\omega(\bar\phi)})^2) = - \log_2((\tfrac{\phi}{\omega(\phi)})^2)$.
This gives the expression for $F_{i, r}(r + 2 x)$ stated in the theorem.
By construction, $F(n) = F_{(n \bmod 3), (n \bmod 2)}(n)$ for all $n \geq 0$.

To see that $F_{i, r}(r + 2 x) \in \Z_2$ if $x \in \Z_2$, take a sequence of integers $(x_m)_{m \geq 0}$ that converges to $x$ such that $r + 2 x_m \equiv i \mod 3$ for each $m \geq 0$; since $F_{i, r}(r + 2 x_m) = F(r + 2 x_m) \in \Z$, it follows by continuity that $F_{i, r}(r + 2 x) \in \Z_p$.
\end{proof}

As in the proof of Theorem~\ref{interpolation p=2}, let
\[
	A_{i, r} = \{n \geq 0 : \text{$n \equiv i \mod 3$ and $n \equiv r \mod 2$}\}.
\]
By the Chinese remainder theorem, $A_{i, r}$ is dense in $r + 2 \Z_2$.
Theorem~\ref{interpolation p=2} implies
\[
	\dens(2)
	= \mu\!\left(\bigcup_{i = 0}^2 \big(F_{i, 0}(2 \Z_2) \cup F_{i, 1}(1 + 2 \Z_2)\big)\right).
\]

An upper bound on $\dens(2)$ can be obtained easily.
Namely, the set of residues modulo~$32$ attained by the Fibonacci sequence is
\[
	\{0, 1, 2, 3, 5, 7, 8, 9, 11, 13, 15, 16, 17, 19, 21, 23, 24, 25, 27, 29, 31\}.
\]
This set has size $21$, so $\dens(2) \leq \frac{21}{32}$.
Proposition~\ref{image p=2} below implies that $\frac{21}{32}$ is also a lower bound on $\dens(2)$.

Write $F_{i, r}(r + 2 x) = g_r(h_{i, r}(x))$ where $g_r(y) = \frac{y - (-1)^r y^{-1}}{\sqrt{5}}$ and
\[
	h_{i, r}(x)
	= \omega(\phi)^{i - r} \phi^r \exp_2\!\left(x \log_2((\tfrac{\phi}{\omega(\phi)})^2)\right).
\]
We next determine $h_{i, r}(\Z_2)$; this is analogous to Lemma~\ref{image under h}.

\begin{lemma}\label{image under h p=2}
For all $i \in \{0, 1, 2\}$ and $r \in \{0, 1\}$, we have $h_{i, r}(\Z_2) = \omega(\phi)^{i - r} \phi^r \exp_2(4 \sqrt{5} \Z_2)$.
\end{lemma}

\begin{proof}
It follows from $\lvert\log_2((\frac{\phi}{\omega(\phi)})^2)\rvert_2 = \frac{1}{4}$ that $x \log_2((\frac{\phi}{\omega(\phi)})^2)$ is in the domain of $\exp_2$ for all $x \in \Z_2$.
Since $(\frac{\phi}{\omega(\phi)})^2 \cdot (\frac{\bar\phi}{\omega(\bar\phi)})^2 = (-1)^2 = 1$, Lemma~\ref{pure square root 5} implies that $\log_2((\frac{\phi}{\omega(\phi)})^2) \in 4 \sqrt{5} \Z_2$.
Therefore
\begin{align*}
	h_{i, r}(\Z_2)
	&= \omega(\phi)^{i - r} \phi^r \exp_2\!\left(\Z_2 \log_2((\tfrac{\phi}{\omega(\phi)})^2)\right) \\
	&= \omega(\phi)^{i - r} \phi^r \exp_2(4 \sqrt{5} \Z_2). \qedhere
\end{align*}
\end{proof}

The following proposition is analogous to Propositions~\ref{Lucas non-zero image} and \ref{Lucas zero measure}.
For the Lucas zero $i = 0$ there is no partial branching, but the images have smaller measure than the images for the Lucas non-zeros.

\begin{proposition}\label{image p=2}
Let $i \in \{0, 1, 2\}$ and $r \in \{0, 1\}$.
The image of $r + 2 \Z_2$ under $F_{i, r}$ satisfies
\begin{align*}
	F_{1, 0}(0 + 2 \Z_2) &\supseteq 3 + 4 \Z_2 \\
	F_{i, r}(r + 2 \Z_2) &\supseteq 1 + 4 \Z_2 \text{ for all $(i, r) \in \{(1, 1), (2, 0), (2, 1)\}$} \\
	F_{0, 0}(0 + 2 \Z_2) &\supseteq 8 \Z_2 \\
	F_{0, 1}(1 + 2 \Z_2) &\supseteq 2 + 32 \Z_2.
\end{align*}
\end{proposition}

\begin{proof}
There are six cases.
Let $z \in 3 + 4 \Z_2$ if $(i, r) = (1, 0)$ and $z \in 1 + 4 \Z_2$ if $(i, r) \in \{(1, 1), (2, 0), (2, 1)\}$.
In these four cases, $z^2 \equiv 1 \mod 8$, so $5 z^2 + (-1)^r 4 \equiv 1 \mod 8$;
therefore $\Z_2$ contains square roots of $5 z^2 + (-1)^r 4$.
Similarly, if $(i, r) = (0, 0)$ and $z \in 8 \Z_2$, then $2 \Z_2$ contains square roots of $5 z^2 + (-1)^r 4 = 5 z^2 + 4 \equiv 4 \mod 64$.
If $(i, r) = (0, 1)$ and $z \in 2 + 32 \Z_2$, then $4 \Z_2$ contains square roots of $5 z^2 + (-1)^r 4 = 5 z^2 - 4 \equiv 16 \mod 128$.

As in the proof of Proposition~\ref{Lucas non-zero image}, $g_r(y) = z$ is equivalent to $y^2 - \sqrt{5} z y - (-1)^r = 0$, the solutions of which are $y = \frac{\pm w + \sqrt{5} z}{2}$, where $w \in \Z_2$ is a square root of $5 z^2 + (-1)^r 4$.
Let $y_0 = \omega(\phi)^{i - r} \phi^r$; then $\overline{y_0} = \omega(\bar\phi)^{i - r} \bar\phi^r$ and $y_0 \overline{y_0} = (-1)^r$.
One checks that in all six cases $g_r(y_0) \equiv z \mod 4$.
Choose $w$ such that $\frac{w + \sqrt{5} z}{2} \equiv y_0 \mod 4$.
Then the conditions of Lemma~\ref{y set general} are satisfied, so $\frac{w + \sqrt{5} z}{2} \in \omega(\phi)^{i - r} \phi^r \exp_2(4 \sqrt{5} \Z_2)$.
By Lemma~\ref{image under h p=2}, there exists $x \in \Z_2$ such that $F_{i, r}(r + 2 x) = z$.
\end{proof}

Proposition~\ref{image p=2} implies that $\dens(2) \geq \frac{1}{4} + \frac{1}{4} + \frac{1}{8} + \frac{1}{32} = \frac{21}{32}$.
Theorem~\ref{theorem p=2} now follows.
In particular, the superset relations for $F_{1, 0}(0 + 2 \Z_2)$, $F_{0, 0}(0 + 2 \Z_2)$, and $F_{0, 1}(1 + 2 \Z_2)$ in Proposition~\ref{image p=2} are equalities.

\section*{Acknowledgment}

We thank Genevieve Maalouf for productive discussions.

\section*{Data availability statement}

Data generated during this study can be found in the OEIS~\cite[\href{http://oeis.org/A350999}{A350999} and \href{http://oeis.org/A351000}{A351000}]{OEIS}.

\end{document}